\documentclass[11pt]{article}
\usepackage{amsmath}
\usepackage{amsthm}
\usepackage{amssymb,amsfonts}
\usepackage[colorlinks=true,citecolor=blue]{hyperref}
\usepackage{latexsym}
\usepackage{todonotes}
\usepackage{nicefrac}
\usepackage{epsfig}
\usepackage{stmaryrd}
\usepackage{setspace}
\usepackage{enumerate}
\usepackage[all]{xypic}
\usepackage{bbm,ifpdf,tikz}
\ifpdf
\usepackage{pdfsync}
\fi

\newtheorem{theorem}{Theorem}[section]

\newtheorem{proposition}[theorem]{Proposition}
\newtheorem{corollary}[theorem]{Corollary}

\newtheorem{conjecture}[theorem]{Conjecture}

{
\theoremstyle{definition}
\newtheorem{definition}[theorem]{Definition}
\newtheorem{example}[theorem]{Example}

\newtheorem{remark}[theorem]{Remark}
}

\usepackage{todonotes}

\newcommand{\cI}{\mathcal{I}}
\newcommand{\rk}{\operatorname{rk}}

\newcommand{\cA}{\mathcal{A}}
\newcommand{\grRep}{\operatorname{grRep}}
\newcommand{\grVRep}{\operatorname{grVRep}}

\newcommand{\la}{\lambda}
\newcommand{\Rep}{\operatorname{Rep}}
\newcommand{\VRep}{\operatorname{VRep}}
\newcommand{\thag}{\text{thag}}
\newcommand{\cP}{\mathcal{P}}
\renewcommand{\cH}{\mathcal{H}}
\newcommand{\cK}{\mathcal{K}}
\newcommand{\cQ}{\mathcal{Q}}
\newcommand{\cT}{\mathcal{T}}
\newcommand{\ch}{\operatorname{ch}}

\begin{document}
\spacing{1.1}
\noindent{\Large\bf Kazhdan-Lusztig polynomials of matroids:\\ a survey of results and conjectures}\\

\noindent{\bf Katie Gedeon, Nicholas Proudfoot\footnote{Supported by NSF grant DMS-1565036.}, 
and Benjamin Young}\\
Department of Mathematics, University of Oregon,
Eugene, OR 97403\\

{\small
\begin{quote}
\noindent {\em Abstract.}
We report on various results, conjectures, and open problems related to Kazhdan-Lusztig polynomials of matroids.
We focus on conjectures about the roots of these polynomials, all of which appear here for the first time.
\end{quote} }


\section{Introduction}
The Kazhdan-Lusztig polynomial of a matroid, introduced in \cite{kl}, is in many ways analogous
to the classical Kazhdan-Lusztig polynomial associated with an interval in the Bruhat poset of a Coxeter group.
In both cases, there is a purely
combinatorial recursive definition.  In the classical setting, the polynomials have a geometric interpretation 
if the Coxeter group is a Weyl group:  they are intersection cohomology Poincar\'e polynomials of certain varieties,
or (equivalently) graded multiplicities of simple objects inside standard objects in a certain category of perverse
sheaves.  In particular, this implies that the coefficients are non-negative.  Non-negativity for arbitrary
Coxeter groups was conjectured by Kazhdan and Lusztig \cite{KL79}, but was only recently proven (35 years
later) by Elias and Williamson \cite{EW14}.

The story for matroids is similar, but still unresolved.  The analogue of a Weyl group is a realizable
matroid.  If a matroid is realizable, then its Kazhdan-Lusztig polynomial is the
intersection cohomology Poincar\'e polynomials of a certain variety, or (equivalently) the graded
multiplicity of a simple object inside of a standard object in a certain category of perverse sheaves. 
In particular, this implies that the coefficients are non-negative.  Non-negativity for arbitrary matroids
is still an open problem (Conjecture \ref{non-neg}).

Despite these analogies, there are important disparities between the two theories.  In the classical
setting, any polynomial with non-negative integer coefficients and constant term 1 arises as a Kazhdan-Lusztig
polynomial (even for the symmetric group $S_n$) \cite{Polo}.  In contrast, Kazhdan-Lusztig polynomials of matroids
appear to be very special.  Experimental evidence suggests that these polynomials are always log concave
and (even better) real rooted (Conjecture \ref{conjecture:real-rooted}).  Furthermore, two matroids that are related to each other 
by a contraction appear to have interlacing roots (Conjecture \ref{conjecture:interlacing} and Remark \ref{P-remark}).
Thus the theory of Kazhdan-Lusztig polynomials of matroids conjecturally contains surprisingly deep
structures that are not present in the classical theory.

We note that both classical Kazhdan-Lusztig polynomials and Kazhdan-Lusztig
polynomials of matroids are special cases of a more general definition introduced by Stanley \cite{Stanley89}
and further developed by Brenti \cite{brenti}.  The matroidal analogue of the $R$-polynomial is the characteristic
polynomial of an interval.  However, we stress that the various properties discussed in this paper, such as positivity
and real rootedness, are special to the case of matroids.  It would be interesting to investigate if there is a natural
level of generality in between ours and Stanley's in which these properties still hold.

Our goal in this paper is to give results and conjectures for arbitrary matroids as well as specific families of examples.  
We will also discuss equivariant Kazhdan-Lusztig polynomials,
introduced in \cite{ekl}, which are finer invariants of matroids with symmetries in which the integer (conjecturally non-negative)
coefficients of the polynomial are replaced by virtual (conjecturally honest) representations of the symmetry group.
In the case of uniform matroids, thagomizer matroids,
and braid matroids, one has an action of the symmetric group, and the coefficients of the equivariant 
Kazhdan-Lusztig polynomial are best understood as (Schur positive) symmetric functions.

\vspace{\baselineskip}
\noindent
{\em Acknowledgments:}
The authors are grateful to Nima Amini, June Huh, Steven Sam, David Speyer, and John Wiltshire-Gordon
for helpful conversations.  All computer calculations were done in SAGE \cite{sage}. 

\section{Definition and positivity}
Let $M$ be a matroid on a finite ground set $\cI$, and let $L(M)$ denote the lattice of flats of $M$, with minimum element $\emptyset$.  
Let $\mu$ be the M\"obius function on $L(M)$,
and let $$\chi_M(t) := \sum_{F\in L(M)} \mu(\emptyset,F)\, t^{\rk M - \rk F}$$ 
be the {\bf characteristic polynomial} of $M$.
For any flat $F\in L(M)$, let $\cI^F = \cI\smallsetminus F$ and $\cI_F = F$.  Let $M^F$ be the matroid
on $\cI^F$ consisting of subsets of $\cI^F$ whose union with a basis for $F$ are independent in $M$, 
and let $M_F$ be the matroid on $\cI_F$ consisting of subsets of $\cI_F$ which are independent in $M$.
We call the matroid $M^F$ the {\bf restriction} of $M$ at $F$, and $M_F$ the {\bf localization}
of $M$ at $F$. This terminology and notation comes from the corresponding
constructions for hyperplane arrangements; the matroid $M^F$ is also known as a {\bf contraction}.
We have $\rk M^F = \rk M - \rk F$ and $\rk M_F = \rk F$.  The lattice of flats of $M^F$ is isomorphic to the portion
of $L(M)$ lying above $F$, while the lattice of flats of $M_F$ is isomorphic to the portion of $L(M)$ lying below $F$
(see Figure \ref{fig:lattice}).

\begin{figure}[h]
\begin{center}
\includegraphics{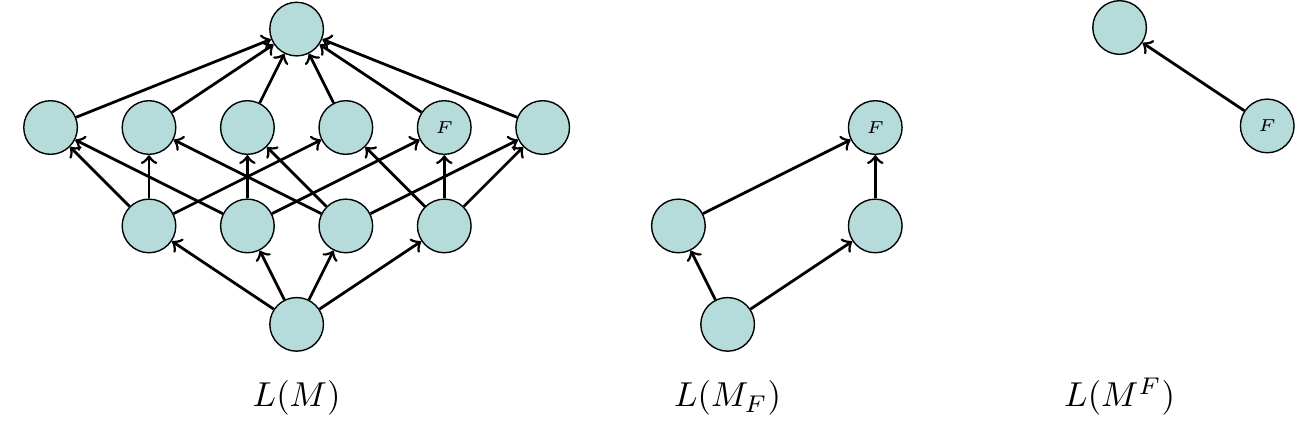}
 \caption{Localization and restriction at a flat of a matroid $M$.}
   \label{fig:lattice}
   \end{center}
\end{figure}

The following theorem is proven in~\cite[Theorem 2.2]{kl}; it is essentially equivalent to the statement that the characteristic
polynomial is a $P$-kernel in the sense of \cite{brenti}.
\pagebreak

\begin{theorem}\label{thm:KL-exists}
There is a unique way to assign to each matroid $M$ a polynomial $P_M(t)\in\mathbb{Z}[t]$ 
such that the following conditions are satisfied:
\begin{enumerate}
\item If $\rk M = 0$, then $P_M(t) = 1$.
\item If $\rk M > 0$, then $\deg P_M(t) < \tfrac{1}{2}\rk M$.
\item For every $M$, $\displaystyle t^{\rk M} P_M(t^{-1}) = \sum_{F}\chi_{M_F}(t) P_{M^F}(t).$
\end{enumerate}
The polynomial $P_M(t)$ is called the {\bf Kazhdan-Lusztig polynomial} of $M$.
\end{theorem}

\begin{conjecture}
\label{non-neg}
The coefficients of $P_M(t)$ are non-negative.
\end{conjecture}

\begin{theorem}
\label{reciprocal}
Conjecture~\ref{non-neg} holds when $M$ is realizable over some field.
\end{theorem}

\begin{remark}
Theorem~\ref{reciprocal} is proved in \cite[Corollary 3.11]{kl}; the idea of the proof is as follows.  Suppose that $M$ is the matroid associated
with a finite collection $\cA$ of vectors in a vector space $V$.  Let $R_\cA$ be the subring
of rational functions on the dual space $V^*$ generated by the reciprocals of the nonzero elements of $\cA$.
The ring $R_\cA$ is called the {\bf Orlik-Terao algebra} of $\cA$, and its prime spectrum $X_\cA := \operatorname{Spec} R_\cA$
is called the {\bf reciprocal plane} of $\cA$.  One can show that the Kazhdan-Lusztig polynomial of $M$
is equal to the intersection cohomology Poincar\'e polynomial of $X_\cA$, and is therefore non-negative.  
The proof works with $\ell$-adic \'etale cohomology
of varieties defined over finite fields; since any matroid that is realizable over some field is realizable over a finite field,
this argument covers all realizable matroids.
\end{remark}

We now briefly survey what is known about the individual coefficients of $P_M(t)$; see \cite[Section 2.3]{kl} for references.
It is easy to prove that the constant term of $P_M(t)$ is always equal to 1.  The linear term of $P_M(t)$ is equal to the number
of coatoms of $L(M)$ minus the number of atoms, which is always non-negative by the hyperplane theorem.
Furthermore, one can show that this number is equal to zero if and only if $L(M)$ is modular, in which case 
{\em all} of the coefficients of positive powers of $t$ vanish.
This is the first piece of evidence that Kazhdan-Lusztig polynomials of matroids form a much more restrictive class
than classical Kazhdan-Lusztig polynomials.  One can also write down explicit general formulas for the quadratic and cubic terms,
but neither one is manifestly positive.  More recently, Wakefield wrote down a general combinatorial formula
for every coefficient \cite[Theorem 5.4]{wakefield}, though again this formula is not manifestly positive.

By definition, the degree of $P_M(t)$ is bounded above by $\lfloor \frac{\rk M - 1}{2}\rfloor$ if $M$ has positive rank, but this bound is not always
achieved.  For example, as we noted above, the degree is zero whenever $L(M)$ is modular.  If $M$ is the direct sum of
two smaller matroids $M_1$ and $M_2$, then $P_M(t) = P_{M_1}(t)P_{M_2}(t)$, which again results in the degree of $P_M(t)$
having smaller than expected degree (unless $\rk M_1$ and $\rk M_2$ are both odd).  We call a matroid $M$ {\bf non-degenerate}
if $\rk M = 0$ or $P_M(t)$ has degree $\lfloor \frac{\rk M - 1}{2}\rfloor$.
A matroid is {\bf regular} if it is realizable over every field.

\begin{conjecture}
\label{degree}
Every connected regular matroid is non-degenerate.
\end{conjecture}

\begin{remark}
A graphical matroid is regular, and it is connected if and only if the corresponding graph is 2-connected.
Conjecture \ref{degree} is already interesting in the graphical case.
\end{remark}

\section{The roots of the Kazhdan-Lusztig polynomial}
For any matroid $M$, Adiprasito, Huh, and Katz recently proved that the absolute values of the 
coefficients of the characteristic polynomial form a log concave sequence with no internal zeros \cite{AHK}.
Experimental evidence led us to make the same conjecture for the Kazhdan-Lusztig polynomial
\cite[Conjecture 2.5]{kl}.

\begin{conjecture}\label{conjecture:log concave}
For every matroid $M$, the coefficients of $P_M(t)$ form a log concave sequence with no internal zeros.  
\end{conjecture}

Further experimentation leads us to strengthen this conjecture as follows.

\begin{conjecture}
\label{conjecture:real-rooted}
For every matroid $M$, all roots of $P_M(t)$ lie on the negative real axis.
\end{conjecture}

Note that real rootedness is much stronger than log concavity; in particular, the characteristic polynomial
of a matroid (as well as the polynomial obtained by taking absolute values of the coefficients) is {\bf not}
in general real rooted.

We give a proof of Conjecture \ref{conjecture:real-rooted} in the simplest nontrivial case.
Let $U_{m,d}$ be the uniform matroid of rank $d$ on $m+d$ elements.  The matroid $U_{0,d}$
is Boolean; in particular, its lattice of flats is modular, and its Kazhdan-Lusztig polynomial is 1.
The next case is $U_{1,d}$, which is isomorphic to the graphical matroid associated with the cycle of length $d+1$.
By \cite[Theorem 1.2(1)]{ukl}, we have
\begin{equation}\label{uoned}
P_{U_{1,d}}(t) = \sum_{i\geq 0}\;\frac{1}{i+1} \binom{d-i-1}{i} \binom {d+1}{i}\; t^i
\end{equation}
for all $d>0$.

\begin{theorem}
All of the roots of $P_{U_{1,d}}(t)$ lie on the negative real axis.
\end{theorem}

\begin{proof}
A sequence of real numbers $\Gamma = \{\gamma_i\}$ is called a \textbf{multiplier sequence} if, for any polynomial 
$f(x)=\sum a_it^i\in\mathbb{R}[x]$ with only real roots, the polynomial $$\Gamma[f(x)] := \sum a_i \gamma_i t^i$$ is either
identically zero or has only real roots.
For any fixed positive integer $d$, the sequence $$\Gamma(d) := \left\{ \frac{1}{(i+1)!(d+1-i)!}\right\}$$
is a multiplier sequence \cite[Lemma 2.5]{zhang16}.  Let 
\[
h_d(t) :=\sum_{i\geq 0}\;\binom{d-i-1}{i}\; t^i;
\]
this polynomial is real rooted \cite[Lemma 3.2]{zhang16}.  The fact that $P_{U_{1,d}}(t)$ is real rooted now follows
from the observation that $P_{U_{1,d}}(t) = (d+1)!\,\Gamma(d)[h_d(t)]$.  Since the coefficients of $P_{U_{1,d}}(t)$ are positive
(including the constant coefficient), it cannot have any non-negative real roots, therefore all of the roots lie on the negative real axis.
\end{proof}

If two matroids are related to each other by a contraction, numerical evidence suggests that the roots 
of their Kazhdan-Lusztig polynomials are related to each other
in a predictable way, which we now describe.
Let $f(t)$ be a polynomial of degree $n$ and $g(t)$ a polynomial of degree $n-1$.
We say that $f(t)$ {\bf interlaces} $g(t)$ if $f(t)$ and $g(t)$ are both real rooted and their roots alternate,
starting with the smallest root of $f(t)$.
For any matroid $M$ of positive rank, let $Q_M(t) := t^{\rk M - 1}P_M(-t^{-2})$.  If Conjecture \ref{conjecture:real-rooted}
is true, then the roots of $Q_M(t)$ are real and symmetrically distributed around the origin.
Given an element $e$ of the ground set of $M$, let $M/e$ denote the contraction of $M$ at $e$.

\begin{conjecture}
\label{conjecture:interlacing}
If $M$ and $M/e$ are both non-degenerate, then $Q_M(t)$ interlaces $Q_{M/e}(t)$.
\end{conjecture}

\begin{remark}\label{P-remark}
If the rank of $M$ is odd, then the degree of $P_M(t)$ is one greater than the degree of $P_{M/e}(t)$,
and Conjecture \ref{conjecture:interlacing} is equivalent to the statement that
$P_M(t)$ interlaces $P_{M/e}(t)$.
If the rank of $M$ is even, then the degree of $P_M(t)$ is equal to that of $P_{M/e}(t)$,
and Conjecture \ref{conjecture:interlacing} is equivalent to the statement that
$tP_{M/e}(t)$ interlaces $P_{M}(t)$.
\end{remark}

\section{Equivariant Kazhdan-Lusztig polynomials}
Let $W$ be a finite group, and let $\grRep(W)$ and $\grVRep(W)$ denote its graded representation ring and graded
virtual representation ring, respectively.
If $W$ acts on a matroid $M$ via permutations of the ground set, 
we can define {\bf equivariant Kazhdan-Lusztig polynomial} $P_M^W(t)\in\grVRep(W)$,
which has the property that, when we take the graded dimension, we recover $P_M(t)$ \cite{ekl}.  
We omit the formal definition here, but we note that the basic idea is to replace
the characteristic polynomial of $M$ with its Orlik-Solomon algebra, which we may interpret as a graded virtual
representation of $W$, and then categorify each of the items of Theorem \ref{thm:KL-exists}.  
Though we are no longer in the theoretical framework of Stanley and Brenti, the existence of the equivariant
Kazhdan-Lusztig polynomial still involves checking an equivariant analogue of the statement that the
characteristic polynomial is a $P$-kernel \cite[Lemma 2.7]{ekl}.  This lemma, which is very easy to prove in the non-equivariant
setting, is surprisingly difficult in the presence of a group action.
Conjecture \ref{non-neg} generalizes to the equivariant setting as follows \cite[Conjecture 2.13]{ekl}.

\begin{conjecture}\label{conj:equivariant positivity}
For any equivariant matroid $W \curvearrowright M$, $P^W_M(t)\in\grRep(W)$.
\end{conjecture}

\begin{remark}
\label{eq-re}
If $M$ is equivariantly realizable over the complex numbers, 
then $P^W_M(t)$ may be identified with the isomorphism class of the intersection
cohomology of the reciprocal plane, and we obtain a proof of Conjecture \ref{conj:equivariant positivity} that is similar
to the proof of Conjecture \ref{non-neg} in the realizable case \cite[Corollary 2.12]{ekl}.
\end{remark}

Uniform matroids constitute an interesting class of equivariant matroids.  The symmetric group
$S_{m+d}$ acts on the uniform matroid $U_{m,d}$.  Though uniform matroids are all realizable,
$U_{m,d}$ is equivariantly non-realizable provided that $m$ and $d$ are both greater than 1.
Thus Remark \ref{eq-re} does not apply to $U_{m,d}$, but Conjecture \ref{conj:equivariant positivity} 
nonetheless holds for this matroid (Corollary \ref{uniform-pos}).
We regard this as a compelling piece of evidence for Conjecture \ref{conj:equivariant positivity} and,
by extension, for Conjecture \ref{non-neg}.

Another interesting class of equivariant matroids is the class of braid matroids.  These are the graphical matroids
associated with complete graphs, and the symmetric group acts by permuting the vertices.  These matroids {\em are}
equivariantly realizable, so Conjecture \ref{conj:equivariant positivity} follows from Remark \ref{eq-re}.
It is still an open problem to determine the representations that appear; we discuss this problem in more detail in the next section.

It is not clear how to generalize Conjecture \ref{conjecture:real-rooted} to the equivariant setting, but we do have
an equivariant analogue of Conjecture \ref{conjecture:log concave}.
The following definition appears in~\cite{ekl}.

\begin{definition} A sequence $(C_0, C_1, C_2, \ldots)$ in $\VRep(W)$
is {\bf log concave} if, for all $i>0$, $C_i^{\otimes 2} - C_{i-1}\otimes C_{i+1} \in \Rep(W)$.
It is {\bf strongly log concave} if, for all $i\leq j\leq k\leq l$ with $i+l=j+k$, $C_j\otimes C_k - C_{i}\otimes C_{l} \in \Rep(W)$.
We call an element of $\grVRep(W)$ (strongly) log concave if its sequence of coefficients is (strongly) log concave.
\end{definition}

\begin{remark}
If $W$ is trivial, then strong log concavity is equivalent to log concavity, which agrees with the usual notion.
If $W$ is non-trivial, then strong log concavity is stronger than log concavity, and only strong log concavity
is preserved under tensor product.\footnote{The proof of this fact has been communicated to us by David Speyer.}
Thus strong log concavity is a more natural notion.
\end{remark}

The following generalization of Conjecture \ref{conjecture:log concave} appears in \cite[Conjecture 5.3(2)]{ekl}, 
and has been checked on a computer for uniform and braid matroids of small rank.

\begin{conjecture}
For any equivariant matroid $W\curvearrowright M$, $P_M^W(t)$ is strongly log concave.
\end{conjecture}

\section{Examples}
We conclude by discussing some specific classes of matroids for which we have various complete or partial results.
The matroids described in Sections~\ref{sec:uniform}~and~\ref{sec:thag} are the only nontrivial
classes of matroids for which a complete description of the Kazhdan-Lusztig polynomial is known.


\subsection{Uniform matroids}
\label{sec:uniform}
Let $C_{m,d,i}$ be the coefficient of $t^i$ in the $S_{m+d}$-equivariant Kazhdan-Lusztig polynomial of the uniform matroid $U_{m,d}$.
For any partition $\la$, let $V[\la]$ be the irreducible representation of $S_{|\la|}$ indexed by $\la$ (corresponding to the Schur function $s[\la]$ under the Frobenius character map).  The following theorem is proved in \cite[Theorem 3.1]{ekl}. 

\begin{theorem}\label{uniform-description}
For all $i>0$,
$${C_{m,d,i}\; 
\;\; = \bigoplus_{b=1}^{\min(m,d-2i)} V[d+m-2i-b+1,b+1,2^{i-1}]\;\in\;\Rep(S_{m+d}).}$$
\end{theorem}

\begin{corollary}\label{uniform-pos}
Conjecture \ref{conj:equivariant positivity} holds for $S_{m+d}\curvearrowright U_{m,d}$.
\end{corollary}

\begin{remark}
We may use the hook length formula for the dimension of $V[\la]$ to compute the
graded dimension of $P^{S_d}_{m,d}(t)$, which is equal to the ordinary Kazhdan-Lusztig polynomial of $U_{m,d}$.
When $m=1$, this formula appears in Equation \eqref{uoned}, and it has a nice combinatorial interpretation:  the $i^\text{th}$
coefficient is equal to the number of ways to choose $i$ disjoint chords in a $(d-i+2)$-gon \cite[Remark 1.3]{ukl}.
In particular, if $d=2n-1$, then the top nonzero coefficient is equal to the $n^{\text{th}}$ Catalan number. 
For arbitrary $m$, this formula is messy and unenlightening.
\end{remark}

\begin{remark}
For $m>1$, we know no way of computing the non-equivariant polynomial of $U_{m,d}$ other than by first computing the equivariant
one and then taking the graded dimension.
\end{remark}

\begin{remark}
For any element $e$ in the ground set of $U_{m,d}$, we have $U_{m,d}/e \cong U_{m,d-1}$.  Thus, by fixing $m$ and
varying $d$, Conjecture \ref{conjecture:interlacing} says that we should obtain an infinite sequence of interlacing
polynomials.  Computer calculations support this conjecture.
\end{remark}

\begin{remark}
If we fix the indices $m$ and $i$ and allow $d$ to vary, we obtain a sequence of representations of larger
and larger symmetric groups.  Once $d$ is greater than or equal to $m+2i$, one can obtain
$C_{m,d+1,i}$ from $C_{m,d,i}$ by adding one box to the first row of each partition appearing in the equation
in Theorem \ref{uniform-description}.  This is a reflection of the fact that the sequence is {\bf representation stable}
in the sense of Church and Farb \cite{CF}, or that it admits the structure of a {\bf finitely generated FI-module}
in the sense of Church, Ellenberg, and Farb \cite{church-ellenberg-farb}.
\end{remark}

The proof of Theorem \ref{uniform-description} involves translating the defining recurrence into the language of symmetric functions, 
rewriting this recurrence as a functional equation for the generating function, and then checking that 
the above representation is a solution.  The functional equation is attractive in its own right, so
we reproduce it here.  Let $\operatorname{ch} C_{m,d,i}$ be the Frobenius character of $C_{m,d,i}$, which is a symmetric function
of degree $m+d$.  Let
$$\cP(t,u,x) := \sum_{m=0}^\infty\sum_{d=1}^\infty\sum_{i=0}^\infty \operatorname{ch} C_{m,d,i}\, x^m u^d t^i.$$
Let $$s(u) := \sum_{n=0}^\infty s[n]u^n,$$ and let
$$\cH(t,u,x) := \frac{u}{u-x}\left(-1 + \frac{s(x)}{s(u)}\right) + \frac{tu}{tu-x}\left(\frac{s(tu)}{s(u)} - \frac{s(x)}{s(u)}\right).$$
Then the defining recurrences for the equivariant Kazhdan-Lusztig polynomials of uniform matroids transform into the following single functional
equation \cite[Equation (2)]{ekl}:
$$\cP(t^{-1},tu,x) = \cH(t,u,x) + \big(1 + \cH(t,u,0)\big)\cP(t,u,x) = \cH(t,u,x) + \frac{s(tu)}{s(u)} \cP(t,u,x).$$
This equation implies an analogous statement for (non-equivariant) exponential generating functions.
Let 
$$P(t,u,x) := \sum_{m=0}^\infty\sum_{d=1}^\infty\sum_{i=0}^\infty \operatorname{dim} C_{m,d,i}\, \frac{x^m u^d t^i}{(m+d)!}$$
and $$H(t,u,x) := \frac{u}{u-x}\left(-1 + e^{x-u}\right) + \frac{tu}{tu-x}\left(e^{tu-u} - e^{x-u}\right).$$
Then we have \cite[Equation (3)]{ekl}:
$$P(t^{-1},tu,x) = H(t,u,x) + \big(1 + H(t,u,0)\big)P(t,u,x) = H(t,u,x) + e^{tu-u}P(t,u,x).$$

\subsection{Thagomizer matroids}
\label{sec:thag}
Consider the complete bipartite graph $K_{2,n}$, and let $T_n$ be the graph obtained by joining the two distinguished vertices with an edge.  The graph $T_n$ is called a {\bf thagomizer graph}.  Let $P_{K_{2,n}}(t)$ and $P_{T_n}(t)$ be the Kazhdan-Lusztig polynomials of the associated graphical matroids, and let 
$c^{\thag}_{n,k}$ be the coefficient of $t^k$ in $P_{T_n}(t)$.
The following theorem is proved in \cite[Theorem 1.1(1)]{thag}.

\begin{theorem}
\label{thm:thag kl poly}
We have
\[c^{\thag}_{n,k} = \frac{1}{n+1}\binom{n+1}{k}\sum_{j=2k}^n \binom{j-k-1}{k-1}\binom{n+1-k}{n-j},\]
the number of Dyck paths of semilength $n$ with $k$ long ascents.  In particular, $P_{T_n}(1)$ is equal to the $n^\text{th}$ Catalan number.
\end{theorem}

We now use Theorem \ref{thm:thag kl poly} to compute the Kazhdan-Lusztig polynomial of $K_{2,n}$.

\begin{theorem}
If $n\geq 2$, then
$P_{K_{2,n}}(t) = P_{T_n}(t) + t$.
\end{theorem}

\begin{proof}
As part of the proof of Theorem \ref{thm:thag kl poly}, one derives the recurrence \cite[Lemma 3.1(1)]{thag}
\[t^{n+1}P_{T_n}(t^{-1}) - P_{T_n}(t) = (t-1)^{n+1} + \sum_{i=0}^{n-1} \binom{n}{i} 2^{n-i}(t-1)^{n-i}P_{T_i}(t).\]
One can show by the same methods that $P_{K_{2,n}}(t)$ satisfies a similar equation:
\begin{eqnarray*}
t^{n+1}P_{K_{2,n}}(t^{-1}) - P_{K_{2,n}}(t) &=& \sum_{i=0}^{n-1} \binom{n}{i} 2^{n-i} (t-1)^{n-i}P_{T_i}(t)\\ &+& \sum_{j=1}^{n} \binom{n}{j} \Big( (t-1)^{j} + (t-1)(t-2)^{j} \Big).
\end{eqnarray*}
By taking the difference, we find that
$$t^{n+1}P_{K_{2,n}}(t^{-1}) - t^{n+1} P_{T_n}(t^{-1}) 
= P_{K_{2,n}}(t) - P_{T_n}(t) +t^n - t.$$
Since $P_{K_{2,n}}(t)$ and $P_{T_n}(t)$ have degree strictly less than $(n+1)/2$, the theorem follows.
\end{proof}

\begin{remark}
If one contracts an edge of $K_{2,n}$ or an edge of $T_n$ (other than the distinguished edge), one obtains $T_{n-1}$.
Thus, Conjecture \ref{conjecture:interlacing} says that both $Q_{K_{2,n}}(t)$ and $Q_{T_n}(t)$ should interlace both $Q_{T_{n-1}}(t)$.
Computer calculations support this conjecture.
\end{remark}

The graph $T_n$ admits an action of $S_n$ that permutes the $n$ non-distinguished vertices.
Let $$\cT(t,u) := \sum_{n=0}^\infty \ch P_{T_n}^{S_n}(t)\, u^{n+1}
\quad\text{and}\quad
v(t,u) := \sum_{n=0}^\infty s[n]\Big[ (t-2) s[1]\Big] u^n,$$
where square brackets denote a plethysm of symmetric functions.
As in the uniform case, the defining recurrence for the equivariant
Kazhdan-Lusztig polynomials transforms into an elegant equation involving power series with symmetric function coefficients \cite[Proposition 4.7(2)]{thag}:
$$\cT(t^{-1},tu) = (t-1)us(u)v(t,u) + \frac{s(tu)^2}{s(u)^2}\cT(t,u).$$
A conjectural solution to this functional equation appears in
\cite[Conjecture 4.1]{thag}.
This functional equation immediately yields an analogous equation for (non-equivariant) exponential generating functions.
Let $$T(t,u) := \sum_{n=0}^\infty P_{T_n}(t)\, \frac{u^{n+1}}{n!} 
= \sum_{n=0}^\infty\sum_{k=0}^\infty c^{\thag}_{n,k}\,\frac{t^ku^{n+1}}{n!}.$$
Then
$$T(t^{-1},tu) = (tu-u)e^{tu-u} + e^{2(tu-u)} T(t,u).$$

\subsection{Braid matroids}
The {\bf braid matroid} $B_n$ is the graphical matroid associated with the complete graph on $n$ vertices,
or (equivalently) with the reflection arrangement associated with the Coxeter group $S_n$.
Surprisingly, we do not even have a conjecture for the Kazhdan-Lusztig polynomial (ordinary or $S_n$-equivariant) of $B_n$.
We now survey some partial results from \cite{kl} and \cite{ekl} and announce some partial results whose proofs will appear in a future paper.

\begin{remark}
For any element $e$ in the ground set of $B_n$, we have $B_n/e \cong B_{n-1}$.  Thus, 
Conjecture \ref{conjecture:interlacing} says that we should obtain an infinite sequence of interlacing
polynomials.  Computer calculations support this conjecture.
\end{remark}

As in the previous two cases, it is helpful to think about the generating functions.
As before, we have one version in which we take the Frobenius characteristic of the equivariant polynomials,
and one version in which we take the exponential generating function for the ordinary polynomials:
$$\cQ(t,u) := \sum_{n=1}^\infty \ch P_{B_n}^{S_n}(t)u^{n-1}\qquad\text{and}\qquad
Q(t,z) := \sum_{n=1}^\infty P_{B_n}(t)\frac{z^n}{n!}.$$
Let $$\cK(t,u) := t^{-1}\left(-1 + \prod_{k=1}^\infty(1+up_k)^{\frac{1}{k}\sum_{d|k}\mu(k/d)t^{d}}\right)
\qquad\text{and}\qquad
K(t,z) = t^{-1}\left(-1 +(1+z)^t\right).$$
Then we have the following functional equations \cite[Equation (7)]{ekl}:
$$\cQ(t^{-1},tu) = u^{-1} \cQ(t,1)\big[\cK(t,u)\big]
\qquad\text{and}\qquad
Q(t^{-1},tz) = t\,Q(t, K(t,z)).$$
The equivariant polynomials up to $n=9$ are given in \cite[Section 4.3]{ekl}.  The non-equivariant
polynomials up to $n=20$ appear in the appendix of \cite{kl},
where the following conjecture was stated.

\begin{conjecture}
The leading coefficient of $P_{B_{2k}}(t)$ is equal to $(2k-3)!!(2k-1)^{(k-2)}$, the number
of labelled triangular cacti on $(2k-1)$ nodes \cite[Sequence A034941]{oeis}.
\end{conjecture}

Let $G_i(z)$ be the coefficient of $t^i$ in $Q(t,z)$, which is the exponential generating function for the $i^\text{th}$ coefficient
of the Kazhdan-Lusztig polynomial of the braid matroid.  
Let $H_i(z)$ be the ordinary (as opposed to exponential) generating function for the $i^\text{th}$ coefficient
of the Kazhdan-Lusztig polynomial of the braid matroid.
The following result is new.

\begin{proposition}\label{egf}
There exist polynomials $p_{ij}(z)$ 
such that $G_i(z) = \sum_{j=0}^{2i} p_{ij}(z)e^{jz}$.
Equivalently, $H_i(z)$ is a rational function 
whose poles are contained in the set 
$\{j^{-1}\mid 1\leq j\leq 2i\}$.
\end{proposition}

\begin{example}
We have
$$H_1(z) = \frac{z^4}{(1-z)^3(1-2z)}\qquad\text{and}\qquad H_2(z) = \frac{15z^6-50z^7 + 40z^8 + 4z^9}{(1-z)^5(1-2z)^3(1-4z)}.$$ 
\end{example}

\begin{remark}
The proof of Proposition \ref{egf} involves showing that the equivariant Kazhdan-Lusztig polynomials of
braid matroids admit the structure of a {\bf finitely generated \boldmath$\operatorname{FS^{op}}$-module}
in the sense of Sam and Snowden \cite{SamSnowden}.  In contrast with FI-modules, which have been studied
extensively, relatively little is known about the behavior of $\operatorname{FS^{op}}$-modules.
\end{remark}

\bibliography{fpsac}
\bibliographystyle{amsalpha}

\end{document}